\documentclass[11pt, english]{article}

\usepackage[margin= 2 cm,bottom=19mm, top= 18mm]{geometry}

\usepackage[dvipsnames]{xcolor}
\usepackage[square,sort,comma,numbers]{natbib}
\usepackage{amsmath}
\usepackage{amsthm}
\usepackage{amssymb}
\usepackage{setspace}
\usepackage{mathtools}
\usepackage{graphicx}
\graphicspath{ {./images/} }
\usepackage[hidelinks]{hyperref}
\usepackage{cleveref}
\usepackage{hyperref}
\usepackage{enumitem}
\usepackage{framed}
\usepackage{subcaption}
\usepackage{xspace}
\usepackage{thmtools} 
\usepackage{thm-restate}
\usepackage{tikz-feynman}
\usepackage{thmtools}
\usepackage{thm-restate}

\usepackage{floatrow}

\usepackage{tikz}
\usepackage{mathdots}
\usepackage{xcolor}
\usepackage{diagbox}
\usepackage{colortbl}
\usepackage[absolute,overlay]{textpos}

\usetikzlibrary{shapes.misc}
\usetikzlibrary{decorations.pathmorphing}

\tikzset{snake it/.style={decorate, decoration=snake}}
\usetikzlibrary{math}

\usetikzlibrary{calc}
\usetikzlibrary{decorations.pathreplacing}
\usetikzlibrary{positioning,patterns}
\usetikzlibrary{arrows,shapes,positioning}
\usetikzlibrary{decorations.markings}

\tikzstyle{edge}=[very thick]
\definecolor{bostonuniversityred}{rgb}{0.8, 0.0, 0.0}
\definecolor{arsenic}{rgb}{0.23, 0.27, 0.29}
\tikzstyle{diredge}=[postaction={decorate,decoration={markings,
		mark=at position .95 with {\arrow[scale = 1]{stealth};}}}]
\tikzset{
    arrow/.style={decoration={markings, mark=at position 0.7 with
    {\fill(-0.09*#1,-0.03*#1) -- (0,0) -- (-0.09*#1,0.03*#1) -- cycle;}}, postaction={decorate}},
    arrow/.default=1
}
\tikzset{
    arow/.style={decoration={markings, mark=at position 1 with
    {\fill(-0.09*#1,-0.03*#1) -- (0,0) -- (-0.09*#1,0.03*#1) -- cycle;}}, postaction={decorate}},
    arow/.default=1
}
\tikzset{
    arrrow/.style={decoration={markings, mark=at position 0.9 with
    {\fill(-0.09*#1,-0.03*#1) -- (0,0) -- (-0.09*#1,0.03*#1) -- cycle;}}, postaction={decorate}},
    arow/.default=1
}

\newcommand{\fitellipsis}[2] 
{\draw [fill=white]let \p1=(#1), \p2=(#2), \n1={atan2(\y2-\y1,\x2-\x1)}, \n2={veclen(\y2-\y1,\x2-\x1)}
    in ($ (\p1)!0.5!(\p2) $) ellipse [ x radius=\n2/2+0cm, y radius=1.1cm, rotate=\n1];
}
\newcommand{\Fitellipsis}[2] 
{\draw [fill=white]let \p1=(#1), \p2=(#2), \n1={atan2(\y2-\y1,\x2-\x1)}, \n2={veclen(\y2-\y1,\x2-\x1)}
    in ($ (\p1)!0.5!(\p2) $) ellipse [ x radius=\n2/2+0cm, y radius=1.4cm, rotate=\n1];
}

\theoremstyle{plain}

\newtheorem*{thm*}{Theorem}
\newtheorem{thm}{Theorem}[section]
\Crefname{thm}{Theorem}{Theorems}

\newtheorem*{lem*}{Lemma}
\newtheorem{lem}[thm]{Lemma}
\Crefname{lem}{Lemma}{Lemmas}

\newtheorem*{claim*}{Claim}

\crefname{claim}{Claim}{Claims}
\Crefname{claim}{Claim}{Claims}

\newtheorem{prop}[thm]{Proposition}
\Crefname{prop}{Proposition}{Propositions}

\Crefname{remar}{Remark}{Remarks}

\crefname{cor}{Corollary}{Corollaries}

\newtheorem*{conj*}{Conjecture}
\newtheorem{conj}[thm]{Conjecture}
\crefname{conj}{Conjecture}{Conjectures}

\Crefname{qn}{Question}{Questions}

\Crefname{obs}{Observation}{Observations}

\Crefname{ex}{Example}{Examples}

\theoremstyle{definition}

\Crefname{prob}{Problem}{Problems}

\Crefname{defn}{Definition}{Definitions}

\theoremstyle{remark}

\newcommand{\remove}[1]{}

\newcommand{\eps}{\varepsilon}

\title{\vspace{-1 cm}
P\'osa rotation through a random permutation}
\date{}

\crefname{enumi}{property}{parts}

\newcommand{\es}{\emptyset}
\newcommand{\mb}[1]{\mathbb{#1}}
\newcommand{\var}{\operatorname{Var}}
\newcommand{\sub}{\subseteq}
\newcommand{\sm}{\setminus}

\author{
Nemanja Dragani\'c\thanks{
Mathematical Institute, University of Oxford, UK. Research supported by SNSF project 217926.\\
\emph{email}: \textbf{nemanja.draganic@maths.ox.ac.uk}.
}
\and
Peter Keevash\thanks{Mathematical Institute, University of Oxford, UK. Research supported by ERC Advanced Grant 883810.\\
\emph{email}: \textbf{keevash@maths.ox.ac.uk}.
}}

\begin{document} 
\maketitle
\begin{abstract}
What minimum degree of a graph $G$ on $n$ vertices guarantees that the union of $G$ 
and a random $2$-factor (or permutation) is with high probability Hamiltonian?
Girão and Espuny D{\'\i}az showed that the answer lies in the interval $[\tfrac15 \log n, n^{3/4+o(1)}]$.
We improve both the upper and lower bounds to resolve this problem asymptotically,
showing that the answer is $(1+o(1))\sqrt{n\log n/2}$.
Furthermore, if $G$ is assumed to be (nearly) regular then we obtain the much stronger bound
that any degree growing at least polylogarithmically in $n$ is sufficient for Hamiltonicity. 
Our proofs use some insights from the rich theory of random permutations 
and a randomised version of the classical  technique of P\'osa rotation
adapted to multiple exposure arguments.
\end{abstract}

\section{Introduction}

A Hamilton cycle in a graph $G$ is a cycle that traverses all vertices;
we call $G$ Hamiltonian if it has such a cycle. This core concept in Graph Theory 
has received significant attention in prominent papers (see~\cite{ajtai1985first,bollobas1987algorithm,chvatal1972note,chvatal1972hamilton,MR3545109,cuckler2009hamiltonian,ferber2018counting,hefetz2009hamilton,jackson1980hamilton,krivelevich2011critical,krivelevich2014robust,kuhn2014proof,kuhn2013hamilton,nash1971edge,posa:76})
and surveys (see~\cite{gould2014recent, MR3727617}). 
Determining whether a graph is Hamiltonian is an NP-complete problem, 
so much research seeks simple conditions that imply Hamiltonicity. 
From the perspective of Extremal Graph Theory it is natural to consider the minimum degree; 
here the seminal result is Dirac's theorem~\cite{dirac1952some}
that any graph with $n \geq 3$ vertices and minimum degree at least $n/2$ is Hamiltonian.
Hamiltonicity is also well-studied in Random Graph Theory, where there are very precise results
for Erdos-R\'enyi graphs $G(n,p)$ and random regular graphs $G_{n,d}$ 
(see the original papers \cite{bollobas1984evolution,komlos1983limit,robinson1994almost,robinson1992almost,robinson2001hamilton,krivelevich2001random,cooper2002random} and the textbooks \cite{bollobas2001random,janson2011random});
in particular, for any $d \in [3,n-1]$ with high probability (whp)  $G_{n,d}$ is Hamiltonian
(and this does not hold for $d=1,2$).

In this paper we consider a problem at the interface of the extremal and random questions mentioned above,
namely Hamiltonicity of a graph $G \cup F$ obtained by superimposing on the same set of $n$ vertices
any graph $G$ of given minimum degree $\delta(G)$ and a random $2$-factor $F \sim G_{n,2}$.
This lies within the general framework of randomly perturbed graphs, where one asks whether the superposition 
of some deterministic graph $G$ and random graph $R$ whp satisfies some specified graph property.
This framework was introduced by Bohman, Frieze and Martin \cite{bohman2003how},
who showed that for $\alpha>0$ there is a constant $C=C(\alpha)$ such that
if $\delta(G) \ge \alpha n$ and $R \sim G(n,p)$ with $p>C/n$ then whp $G \cup R$ is Hamiltonian.
The interesting feature of such results is that the extremal and probabilistic assumptions together
guarantee the property, despite each individually falling far short
(one needs $\delta(G) \ge n/2$ or $p > (\log n + \log \log n + \omega(1))/n$).

Espuny Díaz and Girão~\cite{espuny2023hamiltonicity} considered 
the similar problem where now $R \sim G_{n,d}$ is a random $d$-regular graph.
For $d=1$ they showed that $\delta(G) \ge \alpha n$ with $\alpha > \sqrt{2}-1$ guarantees hamiltonicity,
where the value of $\alpha$ cannot be reduced, whereas for $d=2$ they showed that 
$\delta(G) > n^{3/4 + o(1)}$ is sufficient but $\delta(G) > \tfrac15 \log n$ is insufficient.
Here we resolve this problem asymptotically by improving both bounds.

\begin{thm}\label{thm:min deg}
 For any $\eps>0$ there is $n_0$ so that the following holds for all $n>n_0$.
 Suppose $G$ is an $n$-vertex graph with $\delta(G)\geq (1+\eps)\sqrt{n\log n/2}$. 
 Let $F\sim G_{n,2}$ be a random $2$-regular graph  on the same vertex set as $G$. 
 Then whp  $G\cup F$ is Hamiltonian. 
\end{thm}

\begin{thm}\label{thm:lower bound minimum degree}
  For any $\eps>0$ there is $n_0$ so that the following holds for all $n>n_0$.
  There exists an $n$-vertex graph $G$ with minimum degree $(1-\eps)\sqrt{n\log n/2}$ so that
  if $F\sim G_{n,2}$ is a random $2$-regular graph on the same vertex set as $G$
  then whp $G\cup F$ is not Hamiltonian. 
\end{thm}

We also obtain the following result under the additional assumption
that $G$ is approximately regular, in the sense that $\Delta(G) \le K\delta(G)$ for some constant $K$,
where $\Delta(G)$ denotes the maximum degree of $G$. Here we show that it is sufficient
that the degrees have at least polylogarithmic growth.

\begin{thm}\label{thm:regular}
  Let $G$ be an $n$-vertex graph with $\delta(G)=\omega(\log^3n)$ and $\Delta(G)=O(\delta(G))$.
  Let $F\sim G_{n,2}$  be a random $2$-regular graph on the same vertex set as $G$. 
  Then whp $G\cup F$ is Hamiltonian. 
\end{thm}

The proofs of Theorems \ref{thm:min deg} and \ref{thm:regular}
use a randomised version of the classical technique of P\'osa rotation
adapted to multiple exposure arguments. 
We also rely on the well-developed theory of random permutations,
which one would expect to be relevant, as $F\sim G_{n,2}$ 
can be identified with a uniformly random permutation in $S_n$
conditioned on all cycles having length at least three.
Furthermore, it will be clear from the proofs that our results also hold
if one instead takes $F$ to be the graph of a uniformly random permutation
(ignoring any loops and parallel edges, which are anyway useless for making cycles).

\medskip

\emph{Notation.}
Let $G$ be a graph.
We denote its vertex set by $V(G)$, its edge set by $E(G)$,
its minimum degree by $\delta(G)$ and its maximum degree by $\Delta(G)$.
For $S \sub V(G)$, we let $\Gamma_G(S)$ be the set of vertices
that have a neighbour in $S$; we write $N_G(S) = \Gamma_G(S) \sm S$ 
for the external neighbourhood of $S$.
A $C_\ell$-factor on $V(G)$ is a disjoint union of cycles 
isomorphic to $C_\ell$ that span $V(G)$.
Our asymptotic notation $O(\cdot)$, $o(\cdot)$,
$\Omega(\cdot)$, $\omega(\cdot)$, etc refers to large 
$n:=|V(G)|$,~e.g.~$f(n)=\omega(1)$ means $f(n)\to \infty$ as $n \to \infty$

\section{Random permutations}

In this section we record some well-known properties of random permutations
and deduce the lower bound for our minimum degree result (\Cref{thm:lower bound minimum degree}).
For completeness we include the proofs, just using the straightforward Markov and Chebyshev inequalities
that $\mb{P}(X \ge t\mb{E}X) \le 1/t$ and $\mb{P}((X-\mb{E}X)^2 \ge t \var(X)) \le 1/t$
for any non-negative random variable $X$ and $t>0$.

\begin{lem}\label{lem:permutation properties}
Let $F \sim S_n$ be a random permutation or $F\sim G_{n,2}$ be a random $2$-regular graph on $n$ vertices.
Then for any $\eps \in (0,1/2)$, with probability at least $1-2\eps$ there is a cycle in $F$ of length at least $\eps n$.
Also, if $t=\omega(1)$ as $n\rightarrow \infty$ then whp the number of cycles in $F$ 
of length at most $t$ is $(1+o(1))\log t$;  in particular, whp $F$ contains $(1+o(1))\log n$ cycles.
\end{lem}

\begin{proof}
It suffices to prove the statement when $F \sim S_n$ is a random permutation, 
as with probability $\Omega(1)$ such $F$ has no cycles of length $1$ or $2$,
and then it is conditionally a random $2$-regular graph.

We start with the observation that for any $i \in [n]$,
the length of the cycle in $F$ containing $i$ is uniformly distributed in $[n]$.
To see this, consider revealing the cycle starting at $i$ one vertex at a time.
noting that before we return to $i$ each new vertex is uniformly distributed 
over those not yet seen. Thus for any $k \in \{0,\dots,n-1\}$, the probability that the first return
to $i$ is after $k$ steps is $\tfrac{n-1}{n}  \cdots \tfrac{n-k}{n-k+1} \tfrac{1}{n-k} = \frac{1}{n}$. 

For each $i \in [n]$ let $X_i$ count cycles in $F$ of length $i$.
Then $F$ has $iX_i$ vertices in cycles of length $i$, so $\mb{E}[iX_i]=n \cdot 1/n = 1$, i.e.~$\mb{E}X_i=1/i$.
Let $Y$ count vertices in cycles of length at most $\eps n$.
Then $\mb{E}[Y] = \eps n$, so by Markov's inequality $\mb{P}(Y=n) \le 2\eps$,
which proves the first statement.

For the second statement, let $X$ count cycles of length at most $t$.
Then $\mb{E}[X] = \sum_{i=1}^t 1/i \sim \log t$. We will show $\var X \le \mb{E}X$,
which suffices to complete the proof by Chebyshev's inequality.

We write $X = \sum_{C\in \mathcal{C}}X_C$ 
as a sum over all potential cycles $C$ of length at most $t$
of the indicator variable $X_C$ that $C$ is a cycle in $F$.
Note that $X_C X_{C'} = 0$ for any two distinct cycles $C,C'$ with $C \cap C' \ne \es$.
On the other hand, if $C \cap C' = \es$ there are $(n-|C|-|C'|)!$ permutations
containing both $C$ and $C'$, so $\mb{E}[X_C X_{C'}]=(n-|C|-|C'|)!/n!$.
Thus for any $\ell+\ell' \le n$ we have
\[ \sum_{|C|=\ell} \sum_{|C'|=\ell'} \mb{E}[X_C X_{C'}]
= (\ell-1)! \binom{n}{\ell} (\ell'-1)! \binom{n-\ell}{\ell'} (n-\ell-\ell')!/n! = \frac{1}{\ell \ell'}.\]
We deduce
\begin{align*}
    \var(X ) 
     =\sum_{C,C'\in\mathcal{C}}E[X_CX_{C'}]-(E[X])^2
    \leq E[X]+\sum_{\ell,\ell'\leq t}\frac{1}{\ell \ell'}-(E[X])^2   = E[X]. &\qedhere
\end{align*}
\end{proof}

Now we analyse the lower bound construction for our minimum degree result.

\begin{proof}[Proof of \Cref{thm:lower bound minimum degree}]
Fix any $\eps>0$ and let $G$ be a complete bipartite graph with parts $A$ and $B$, where $|A|=(1-\eps)\sqrt{\tfrac12 n\log n}$ and $|B|=n-|A|$.
Our aim is to show that whp the induced subgraph \( F[B]=(G \cup F)[B] \) contains more than \( |A| \) connected components.
This will imply that \( G \cup F \) cannot have a Hamiltonian cycle, as by removing the vertices of \( A \) from a presumed Hamiltonian cycle, 
we observe that the cycle would induce at most \( |A| \) components on \( B \), which is impossible.

We start by considering cycles in $F$ that are completely contained in $B$.
We let $t = \sqrt{n}/\log n$ and note that for any given cycle $C$ in $F$ of length at most $t$
we have $\mb{P}(C \cap A \ne \es) \le t|A|/n = o(1)$. Combining this observation with  \Cref{lem:permutation properties} and Markov's inequality,
we can choose $K = (1/2-o(1))\log n$ so that whp $F$ has at least $K$ cycles of length at most $t$ completely contained in $B$.

Now we consider the collection $\mathcal{C}$ of cycles in $F$ that contain vertices in $A$.
Let $Z$ count all edges of cycles in $\mathcal{C}$ that are contained in $A$.
There are precisely $|A|-Z$ connected components in $B$ induced by cycles in $\mathcal{C}$,
so it suffices to show whp $Z < K$.
We have $\mb{E}[Z]=n\tfrac{|A|(|A|-1)}{n(n-1)} \sim (1-\eps)^2 \tfrac12 \log n$ and 
\[ \mb{E}[Z(Z-1)] = n\tfrac{|A|(|A|-1)(|A|-2)}{n(n-1)(n-2)} + n(n-3) \tfrac{|A|(|A|-1)(|A|-2)(|A|-3)}{n(n-1)(n-2)(n-3)}
= (1+o(1))(\mb{E}[Z])^2, \]
so by  Chebyshev's inequality whp $Z < K$, as required.
\end{proof}

\section{Some applications of concentration}

Here we state some concentration inequalities and deduce some lemmas that will be used later in the paper. 
We will often need the Chernoff bound for binomial and hypergeometric random variables (see e.g.~\cite[Theorems~2.1~and~2.10]{janson2011random}). 

\begin{thm}[Chernoff bounds]
Let $X$ be a binomial or hypergeometric random variable and $\delta>0$. Then
$\mb{P}[ |X-\mb{E}X| \ge \delta \mb{E}X ] \le 2\exp(-\delta^2\mb{E}X/(2+\delta))$.
\end{thm}

For future reference we note an easy consequence on the minimum degree of random induced subgraphs.

\begin{lem}\label{lem:mindegree}
Let $G$ be an $n$-vertex graph with $\delta(G) = d \geq 20\log n$.
Let $X$ be a random subset of $V(G)$ of size $r$, where $n/4\geq r\geq 20 \tfrac{n}{d} \log n$.  
Then $\mb{P}[\ \delta(G-X) \ge (1-2r/n) d\ ] \ge 1-o(n^{-5})$.
\end{lem}
\begin{proof}
For each vertex $v\in V(G)$, the number $d(v,X)$  of neighbours of $v$ in $X$ is hypergeometric 
with $\mb{E}[d(v,X)] = d_G(v) \cdot r/n \ge 20n\log n$.
By Chernoff bounds, we have $\mb{P}(d(v,X) > 2d_G(v) r/n) \le 2e^{-\mb{E}[d(v,X)]/3} = o(n^{-6})$.
The lemma follows by a union bound over $v$.
\end{proof}

We also need McDiarmid's bounded differences inequality,
which follows from Azuma's martingale inequality 
(see e.g.~\cite[Corollary~2.27~and~Remark~2.28]{janson2011random}). 
Suppose $f:S \to \mb{R}$ where $S = \prod_{i=1}^n S_i$
and $b = (b_1,\dots,b_n) \in \mb{R}^n$.
We say that $f$ is \emph{$b$-Lipschitz} if for any 
$s,s' \in S$ that differ only in the $i$th coordinate
we have $|f(s)-f(s')| \le b_i$. 
We also say that $f$ is \emph{$v$-varying} 
where $v=\sum_{i=1}^n b_i^2/4$.

\begin{thm}[Bounded differences inequality]\label{thm:azuma}
Suppose $Z = (Z_1,\dots,Z_n)$ is a sequence 
of independent random variables.
Let $X=f(Z)$, where $f$ is $v$-varying.
Then $\mb{P}(|X-\mb{E}X|>t) \le 2e^{-t^2/2v}$.
\end{thm}

We deduce the following lemma on expansion between random sets.

\begin{lem}\label{lem:expansion}
    Let $G$ be an $n$-vertex graph with $\delta(G)=d$ and $\Delta(G)=O(d)$.
    Let $S$ and $T$ be uniformly random disjoint sets of sizes $s$ and $t$,
    where $s=\omega(\log n)$ and $t = \omega(\tfrac{n}{d}\log n)$.
    Let $Y$ count vertices in $T$ that have a neighbour in $S$.
    Suppose $d':=\min\{d,n/s\} = \omega(\log n)$.
    Then $\mb{P}[Y < \tfrac{d'st}{16n}] < e^{-\omega(\log n)}$.
\end{lem}
\begin{proof}
Let $\Gamma(S)$ be the set of all vertices that have a neighbour in $S$.
Let $E$ be the event that $|\Gamma(S)| > sd'/4$.
We claim that $\mb{P}[E] \ge 1-e^{-\omega(\log n)}$.
This will suffice to prove the lemma,
as $\mb{E}[Y \mid E] \ge t \tfrac{|\Gamma(S)|-s}{n-s} \ge \tfrac{d'st}{8n}$ for $d' \ge 8$,
so by Chernoff bounds $\mb{P}[Y < \tfrac{d'st}{16n} \mid E] \le e^{-\Theta(sd't/n)}\leq e^{-\omega(\log n)}$.

To prove the claim, we let $G'$ be a random subgraph of $G$ where each edge 
is retained with probability $\min\{2d'/d,1\}$.
By Chernoff bounds, $\delta(G') \ge d'$ and $\Delta(G')=O(d')$ with probability $1-e^{-\omega(\log n)}$.
We consider a random subset $S'$ in $V(G')$ of size at most $s$
obtained from $s$ independent samples of a uniformly random vertex of $G'$ (allowing repetitions). 
We let $E'$ be the event that $|\Gamma_{G'}(S')| > sd'/4$ and note that $\mb{P}[E] \ge \mb{P}[E']$.
For each $v \in V(G')$, we have 
\[ \mb{P}(v \notin \Gamma_{G'}(S')) \le (1-d'/n)^s \le 1- sd'/n + \tbinom{s}{2}(d'/n)^2 \le 1 - \tfrac{sd'}{2n}.\]
Thus $\mb{E}[|\Gamma_{G'}(S')|] \ge sd'/2$.
Also, $|\Gamma_{G'}(S')|$ is a $\Delta(G')$-Lipschitz function on the probability space of $S'$, so is $O(sd'{}^2)$-varying,
and so by Azuma's inequality $\mb{P}(|\Gamma(S')| \le sd'/4) \le e^{-\Omega(s)} \le  e^{-\omega(\log n)}$.
\end{proof}

Our final concentration inequality is a variation on Talagrand's convex distance inequality,
adapted to the setting of random permutations by McDiarmid;
we state a simplified form of \cite[Theorem 4.1]{mcdiarmid2002concentration}.

\begin{lem}[Talagrand's inequality for permutations] \label{lem:tal}
Let $\pi \sim S_n$ be a uniformly random permutation and $X=X(\pi)$ be a non-negative random variable.
Suppose that $X$ is $b$-Lipschitz, meaning that $|X(\pi)-X(\pi')| \le bd(\pi,\pi')$, where $d(\pi,\pi')$ denotes Hamming distance.
Suppose also that $X$ is $r$-certifiable, meaning that if $X(\pi)=s$ then there is a set $S$ of $rs$ coordinates
such that $X(\pi') \ge s$ for any $\pi'$ agreeing with $\pi$ on $S$. Let $M$ be a median of $X$ and $t>0$.
Then $\mb{P}[ |X-M| \ge t ] \le 6\exp \left( -\tfrac{t^2}{16rb^2 (M+t)} \right)$.
\end{lem}

We deduce a lemma on hitting two sets by consecutive pairs in a random permutation $\pi$.

\begin{lem} \label{lem:hitpairs}
Let $R,S,T$ be subsets of $[n]$ with $|R| = \Omega(n)$ and $|S||T|=\omega(n)$.
Let $\pi$ be a random permutation and $X=X_{RST}(\pi)$ count pairs $(s,t) \in S \times T$
of the form  $(\pi(i),\pi(i+1))$ or $(\pi(i+1),\pi(i))$ with $i \in R$, interpreting $n+1$ as $1$.
Then whp $X \ge (2-o(1)) |R||S||T|/n^2$.
\end{lem}

\begin{proof}
For each $i \in R$ and  $(s,t) \in S \times T$ we have 
$\mb{P}( (\pi(i),\pi(i+1))=(s,t) ) = \mb{P}( (\pi(i+1),\pi(i))=(s,t) )  = \tfrac{1}{n(n-1)}$,
so $\mb{E}[X] = 2|R||S||T|/n(n-1) = \omega(1)$. 
For concentration, let $\eps=\eps(n)$ be $o(1)$ and $\omega(n/|S||T|)$.
Noting that $X$ is $4$-Lipschitz and $2$-certifiable, we have $M \ge (1-\eps)\mb{E}[X]$,
otherwise Talagrand's inequality with $t = \eps \cdot \mb{E}[X]$ would give
the contradiction $\mb{P}[X \le M] \le e^{-\Theta(\mb{E}[X])} = o(1) < 1/2$.
Applying Talagrand's inequality again gives
$\mb{P}[X \le (1-2\eps)\mb{E}[X]] \le e^{-\Theta(\mb{E}[X])} = o(1)$, as required.
\end{proof}

\section{Asymptotically optimal minimum degree}

In this section we show whp hamiltonicity of $G \cup F$ with $F\sim G_{n,2}$
under the asymptotically optimal minimum degree assumption on $G$,
namely  $\delta(G)=d=(1+\eps)(\sqrt{\tfrac{n}{2}\log n})$,
where $\eps>0$ is small and $n>n_0(\eps)$ is large.
We condition on the lengths of the cycles of $F$, 
fix a $2$-factor $F^*$ on $V(G)$ with precisely these cycle lengths,
and write $F = \pi(F^*)$ where $\pi \in S_n$ is a uniformly random permutation.
We employ a multiple round exposure of the randomness in $\pi$.
The proof strategy is to grow a path $P$ in which $o(n)$ vertices have been exposed
to find connecting edges in $G$ but where most of $P$ remains random, 
depending on the unrevealed randomness in $\pi$;
we use the unrevealed randomness to find more connecting edges. 
We grow $P$ by incorporating cycles of $F$, using the following two phase process:
\begin{enumerate}
\item We call a cycle in $F$ \emph{long} if it has length at least $L := \sqrt{n}\log^2 n$.
Suppose that the current path $P$ starts at $a$ and ends at $b$.
If some long cycle $C$ is not yet incorporated we expose $\sqrt{n} \log n$ consecutive vertices of $C$,
which is a negligible fraction of $C$ yet enough to whp discover a neighbour of $b$,
so that we can use $C$ to extend $P$. Thus by exposing $o(n)$ vertices, 
we can find edges in $G$ that can be used to 
connect up all long cycles in $F$  to form a path $P$ of length $n-o(n)$. 
\item Given the current path $P$, say starting at $a$ and ending at $b$,
to incorporate some short cycle $C$, we fix two consecutive vertices $a_C, b_C$ on $C$, 
then find an edge $xy$ of $P$ such that (i)  $xb$, $ya_C$ are edges of $G$,
and (ii) $x$ appears before $y$ on $P$.
Then we can extend $P$ to a path from $a$ to $b_C$ incorporating $C$,
replacing $b$ by $b_C$ and leaving $a$ unchanged (see Figure \ref{fig:one rotation}). 
When all short cycles have been incorporated, so that $P$ is now a Hamiltonian path from $a$ to some $b$,
we find $xy$ as above so that $xb$, $ya$ are edges of $G$, which gives a Hamiltonian cycle. 
\end{enumerate}
The delicate part of the argument is conserving enough randomness in $P$ 
to find the pairs $xy$ needed in Phase 2. For any short cycle $C$, 
by Talagrand's inequality (applied via \Cref{lem:hitpairs})
we have at least $\sim 2d^2/n \sim (1+\eps)^2 \log n$ ordered edges $xy$ on $P$ such that $xb$, $ya_C$ are edges of $G$.
About half of these should appear in the correct order on $P$, 
whereas we have about $\tfrac12 \log n$ short cycles by \Cref{lem:permutation properties},
so this suggests that we will have enough ordered edges $xy$
if we can control how many became unavailable due to previous short cycles.
We will find a set $M_C$ of at least $(1+\eps)\log n$ suitable ordered edges $xy$ for each $C$
such that $M = \bigcup_C M_C$ is a matching and the conditional distribution of $\pi$
has all $2^{|M|}$ possible orders on $P$ of each edge of $M$ being equally likely.
This will suffice, although some extra care in the argument is needed to handle \emph{twin} edges $xy$
for which both ordered edges $xy$ and $yx$ are in $M_C$.

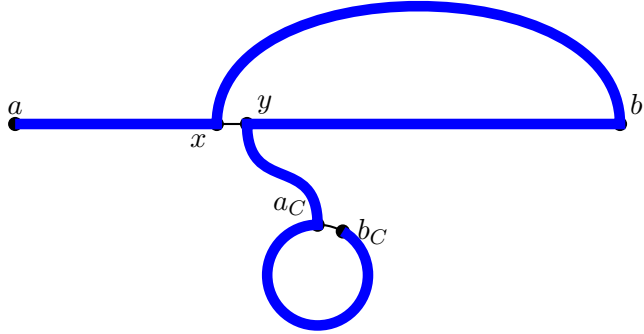
\begin{figure}[ht]
    \centering
    \begin{tikzpicture}[scale=1.34,
        main node/.style={circle,draw,color=black,fill=black,inner sep=0pt,minimum width=3pt}]
        
        \tikzset{cross/.style={
            cross out, 
            draw=black, 
            fill=none, 
            minimum size=2*(#1-\pgflinewidth),
            inner sep=0pt,
            outer sep=0pt},
        cross/.default={2pt}}
        
        \tikzset{rectangle/.append style={
            draw=brown, 
            ultra thick, 
            fill=red!30}}
      
        \coordinate (A)  at (0,0);   
        \coordinate (B)  at (6,0);   
        \coordinate (C)  at (2,0);   
        \coordinate (C') at (2.3,0); 
        \coordinate (D)  at (3,-1.5);
        
        \coordinate (F)  at (3,-1);  
        
        \node[main node] (F') at ($(3,-1.5)+(60:0.5)$) [label=right:$b_C$] {};
        
        \draw[thick] (A) -- (B);
        
        \draw[thick] (C) to[out=90, in=90, looseness=1] (B);
        
        \draw[thick] (C') to[out=-90, in=90, looseness=1.5] (F);
        
        \draw[thick] (D) circle [radius=0.5];
        
        \fill (A) circle (2pt);  
        \fill (B) circle (2pt);  
        \fill (C) circle (2pt);  
        \fill (C') circle (2pt); 
        \fill (F) circle (2pt);  
        \fill (F') circle (2pt);
        \node[above]      at (A)  {$a$}; 
        \node[above right]at (B)  {$b$};
        \node[above right]at (C') {$y$};
        \node[below left] at (C)  {$x$};
        \node[above left] at (F)  {$a_C$};
        
        \draw[blue, line width=4pt, opacity = 0.25]
            (A)
            -- (C)
            to[out=90,in=90,looseness=1] (B)
            -- (C')
            to[out=-90,in=90,looseness=1.5] (F)
            arc [
                start angle=90,
                end angle=420,  
                radius=0.5,
            ];     
    \end{tikzpicture}
    \caption{Using a good pair to incorporate a short cycle $C$ in the current path $P$. }
    \label{fig:one rotation}
\end{figure}

\begin{proof}[Proof of \Cref{thm:min deg}]
Let $G$ be an $n$-vertex graph with $\delta(G)=d=(1+\eps)(\sqrt{n\log n/2})$,
where $\eps>0$ is small and $n>n_0(\eps)$ is  large.
Let $F\sim G_{n,2}$ be a random $2$-regular graph on the same vertex set as $G$. 
We want to show that whp $G\cup F$ is Hamiltonian. 

By \Cref{lem:permutation properties}, whp $F$ has $(1+o(1))\log n$ cycles,
of which $(\tfrac12+o(1)) \log n$ are short (of length $<L := \sqrt{n}\log^2 n$).
We condition on such a set of cycle lengths for $F$,
fix a $2$-factor $F^*$ on $V(G)$ with precisely these cycle lengths,
and write $F = \pi(F^*)$ where $\pi \in S_n$ is a uniformly random permutation.

As described before the proof, we will grow a path $P$ in which $o(n)$ vertices 
have been exposed to find connecting edges in $G$ but where most of $P$ remains random, 
depending on the unrevealed randomness in $\pi$. To start with, we fix any vertex of $F^*$, say $1$,
expose $b=\pi(1)$ and let $P$ be a random path ending at $b$ that follows the image under $\pi$
of the cycle in $F^*$ containing $1$. By exposing $b$, we mean that $b$ is fixed and $\pi$ is now
a uniformly random permutation conditioned on $\pi(1)=b$. 

At any point in the proof, 
we will have some exposed set $B = \{b_i: i \in I\} \sub V(G)$ where $I \sub V(G)$ with $|I|=o(n)$,
such that $B$ is a uniformly random set of size $|I|$, and the conditional distribution of $\pi$
given $B$ is uniformly random conditioned on $\pi(i)=b_i$ for all $i \in I$.
By Chernoff bounds for hypergeometric distributions, 
whp every vertex has $o(d)$ exposed neighbours.
We will always denote the current path by $P$ and its final vertex by $b$ 
(so $b$ will change each time we incorporate a cycle).
The first vertex of $P$, which we denote by $a$, will remain the same throughout the proof.

\medskip

\emph{Phase 1: incorporating long cycles.}

\medskip

Given the current path $P$ from $a$ to some $b$,
if some long cycle $C$ is not yet incorporated 
we expose consecutive $\sqrt{n} \log n$ vertices of $C$.
As $b$ had at least $d-o(d)$ unexposed neighbours,
the number of neighbours of $b$ exposed on $C$
is hypergeometric with mean 
at least $\sim d (\sqrt{n} \log n)/n = \Omega(\log^{3/2} n)$.
By Chernoff bounds, with probability $1-o(1/n)$ 
some neighbour of $b$ is thus exposed,
so $P$ can be extended to incorporate $C$.
We write $P_0$ for $P$ after Phase 1 and let $b_0 = b$ be the last vertex of $P_0$.
Then $P_0$ has length $n-o(n)$, as the total length of the short cycles is $o(n)$.

\medskip

\emph{Preparation for Phase 2: counting good edges.}

\medskip

We enumerate the short cycles as $C_1,\dots,C_k$ for some $k \sim \tfrac12 \log n$.
We fix and expose two consecutive vertices $a_i, b_i$ on $C_i$ for each $i \in [k]$.
We will incorporate the short cycles in order, so that after incorporating any $C_i$
the current path $P$, denoted by $P_i$, incorporates $C_1,\dots,C_i$ and ends at $b_i$.
We let $a_{k+1}=a$ be the start vertex of $P$.
For $i \in [k+1]$, we say that an ordered pair $xy$ is $i$-good 
if $x$ and $y$ are unexposed and $x b_{i-1}$ and $ya_i$ are edges of $G$.
If $xy$ is $i$-good then we will be able to use it to incorporate $C_i$
(or convert a Hamilton path to a Hamilton cycle if $i=k+1$) 
if (a) $x$ appears before $y$ on $P_{i-1}$, and 
(b) neither of $x,y$ was used to incorporate any $C_j$ with $j<i$.

For each $i \in [k+1]$, we let $M_i=M_i(\pi)$ be the set of $i$-good ordered edges $xy$ of $P_0$ 
such that $xy$ and possibly $yx$ are the only ordered pairs that are $i'$-good for some $i' \in [k+1]$ 
within the $4$ consecutive vertices $w,x,y,z$ of $P_0$ surrounding $xy$.
Thus the set $M$ of unordered pairs appearing in any $M_i$ is a matching contained in $P_0$.
For any $xy \in M$, transposing $x$ and $y$ in $\pi$ does not affect any $M_i$,
so the conditional distribution of $\pi$ given $M_1,\dots,M_k$ 
makes all orders on $P_0$ of the pairs in $M$ equally likely.

We claim that whp each $M_i$ contains all $i$-good edges.
To see this, we bound the number of discarded
$i$-good pairs by the number of $Y_i$ of $i$-good pairs $xy$, surrounded in $P_0$ by some $w$ and $z$,
such that $w$ or $z$ is an unexposed neighbour of $a_{i'}$ or $b_{i'}$ for some $i' \in [k+1]$.
There are $O(\log^2 n)$ choices for $(i,i')$, and $O(d^3)$ choices for $x,y$ and then one of $w$ or $z$,
having at most $O(n)$ choices for their position in $P_0$, 
each occurring with probability $O(n^{-3})$,
so $\sum_i \mb{E}[Y_i] = O(d^3 n^{-2} \log^2 n) = o(1)$. 
Thus whp all $Y_i=0$, which proves the claim.

We also claim that each $|M_i|>(1+\eps)\log n$.
To see this we apply Lemma \ref{lem:hitpairs}
with $R$ equal to the set of $j \in [n]$ such that $\pi(j-1),\pi(j),\pi(j+1)$ are all unexposed,
and $S,T$ respectively equal to the set of unexposed vertices in $N_G(b_{i-1})$, $N_G(a_i)$,
deducing that whp for each $i$
at least $(2-o(1)) |R||S||T|/n^2 \ge (1-o(1))(1+\eps)^2 \log n$ pairs are $i$-good.
Here we can take a union bound over $O(\log n)$ choices of $i$,
as by the proof of Lemma \ref{lem:hitpairs} the failure probability
for each $i$ decays exponentially in $\log n$.
Strictly speaking, we cannot apply Lemma \ref{lem:hitpairs} directly to $\pi$, 
as after Phase 1 we have exposed $o(n)$ vertices $B = \{b_i: i \in I\} \sub V(G)$
and the conditional distribution of $\pi$ given $B$ is uniformly random conditioned on $\pi(i)=b_i$ for all $i \in I$.
Thus we fix some $\pi_0$ with $\pi_0(i)=b_i$ for all $i \in I$, write $\pi = \sigma \pi_0$ where $\sigma$ is
a uniformly random permutation of $V(G) \sm B$, and apply Lemma \ref{lem:hitpairs} to $\sigma$.
The claim follows.

\medskip

\emph{Phase 2: incorporating short cycles.}

\medskip

We will incorporate the short cycles in order, as described above,
so that after incorporating any $C_i$ with $i \in [k]$
the current path $P$, denoted by $P_i$, incorporates $C_1,\dots,C_i$ and ends at $b_i$.
In step $k+1$ the same method will convert a Hamiltonian path to a Hamiltonian cycle.

We say that an unordered pair $\{x,y\}$ is an \emph{$i$-twin} 
if both ordered pairs $xy$ and $yx$ are $i$-good.
We let $T$ be a maximal collection of twins that can be assigned to distinct 
short cycles:~formally, let $T \sub M$ be maximal such that each $xy \in T$
is an $f(xy)$-twin for some injection $f:T \to [k+1]$.
For each step $i \in [k]$ of the form $f(xy)$ for some $xy \in T$,
we can incorporate $C_i$ using the $i$-good pair $xy$ or $yx$
according to whichever order is prescribed by $P_{i-1}$.
We can assume $|T| \le k \sim \tfrac12 \log n$,
otherwise all steps can be completed using twins.

For each $i \in [k+1] \sm f(T)$, by maximality of $T$
there are no $i$-twins disjoint from $T$. For such $i$, 
we maintain a subset $M'_i \sub M_i$ 
of ordered pairs that have not been discarded,
where initially we discard all pairs that are in $T$ (ignoring order).
We greedily expose the order of pairs in $M'_i$
until we find one with the correct order is prescribed by $P_{i-1}$;
we also discard all exposed pairs from all $M'_{i'}$.

To complete the proof, it suffices to show that whp no $M'_i$ becomes empty.
Initially, after discarding $T$, each $|M'_i| \ge N := (1+\eps)\log n - 2|T|$,
so it suffices to show that whp we expose fewer than $N$ pairs.
Suppose that we expose $N$ pairs and
let $X$ count successful exposed pairs, meaning that they were in the correct order.
Each exposed pair has probability $1/2$ of success independently
of all others, so $X$ is binomial with mean $N/2$.
To complete the process we require $k+1 - |T|$ successes,
so $X \le k+1 - |T| = (\tfrac12 + o(1))\log n - |T|$.
By Chernoff bounds, this event has probability $o(1)$, as required. 
\end{proof}

\section{Approximately regular graphs}

Here we show whp hamiltonicity of $G \cup F$ with $F \sim G_{n,2}$
when $G$ is an $n$-vertex graph with minimum degree $\delta(G)=\omega(\log^3n)$ 
and maximum degree $\Delta(G)=O(\delta(G))$.
As before, we condition on the lengths of the cycles of $F$, 
fix a $2$-factor $F^*$ on $V(G)$ with precisely these cycle lengths,
and write $F = \pi(F^*)$ where $\pi \in S_n$ is a uniformly random permutation.

By \Cref{lem:permutation properties} we can assume that $F^*$ has $O(\log n)$ cycles
and at least one cycle $C_1$ of length $\Omega(n)$, which will play the role of the path $P_0$
in the proof of \Cref{thm:min deg}:~we will expose vertices of $\pi(C_1)$ to discover edges of $G$
that can be used to incorporate the other cycles of $F$. However, the degrees in $G$ may be only
polylogarithmic, which is much sparser than our previous setting, so we cannot hope to find the two edge
connections used before. Instead we work with the following definition from \cite{draganic2022pancyclicity},
illustrated in Figure \ref{fig:special set}: given a graph on an ordered vertex set,
we say that a sequence of vertices $v_1 < \ldots < v_{\ell}$ is \emph{special} 
if $(v_{i},v_{i+1}^+)$ is an edge for all $i\in [\ell-1]$,
where for any $v$ we write $v^-$ or $v^+$ for the predecessor or successor of $v$ in the vertex order.
If also $v_1 v_1^+ \dots v_\ell^- v_\ell$ is a path $P$ then we can use these edges
to form a \emph{special path} from $v_\ell^-$ to $v_\ell$ that uses all vertices of $P$;
this is illustrated in the $I_1$ part of \Cref{fig:rotation}.

\begin{figure}[ht]
    \centering
    \begin{tikzpicture}[scale=1.34,main node/.style={circle,draw,color=black,fill=black,inner sep=0pt,minimum width=3pt}]
        \tikzset{cross/.style={cross out, draw=black, fill=none, minimum size=2*(#1-\pgflinewidth), inner sep=0pt, outer sep=0pt}, cross/.default={2pt}}
	\tikzset{rectangle/.append style={draw=brown, ultra thick, fill=red!30}}

	    \foreach \i in {1,...,47}
	    {
	        \node[main node, scale=0.4] (aux) at (\i*0.2-0.6,0){};
	    }
	    \node[rectangle, color=red, scale=0.4] (a) at (0,0) [label=below:$v_1$]{};
	    \node[rectangle, color=red, scale=0.4] (a1) at (2,0)[label=below:$v_2$]{};
	    \node[rectangle, color=red, scale=0.4] (a2) at (3,0)[label=below:$v_3$]{};
	    \node[rectangle, color=red, scale=0.4] (a3) at (5.4,0)[label=below:$v_4$]{};
	     \node[rectangle, color=red, scale=0.4] (a4) at (8,0)[label=below:$v_5$]{};
	    \node[main node] (a4) at (8.2,0){};
	    
	    \node[main node] (b1) at (2.2,0){};
	    \node[main node] (b2) at (3.2,0){};
	    \node[main node] (b3) at (5.6,0){};
	    
	    \draw[line width= 1 pt] (a) to [bend left=60](b1);
	    \draw[line width= 1 pt] (a1) to [bend left=60](b2);
	    \draw[line width= 1 pt] (a2) to [bend left=60](b3);
	    \draw[line width= 1 pt] (a3) to [bend left=60](a4);


    \end{tikzpicture}
    \caption{A sequence $v_1 < \ldots < v_5$ of special vertices.}
    \label{fig:special set}
\end{figure}
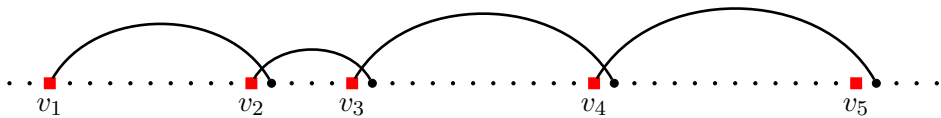

The following lemma will be used to incorporate cycles into the path;
the intended use (see \Cref{fig:rotation}) is that $x$ is the start vertex,  
$v_1 \dots v_\ell$ is some unexposed interval,
and $y$ is a vertex of the cycle $C$ to be incorporated,
after which the new path will start at a vertex adjacent to $y$ on $C$. 

\begin{lem} \label{lem:special}
Let $G$ be an $n$-vertex graph with minimum degree $\delta(G)=d=\omega(\log^3n)$ 
and maximum degree $\Delta(G)=O(d)$. 
Fix any $\{x,y\} \sub V(G)$ and $k = \Theta(n/\log n)$.
Consider an ordered subgraph $G'$ on
uniformly random distinct vertices $v_1 < \dots < v_k$ in $V(G) \sm \{x,y\}$.
Then with probability $1-o(n^{-3})$ there is a special vertex sequence
$u_1 < \dots < u_\ell$ in $G'$ with $\ell < \log n$ such that
$xu_1^+$ and $yu_\ell^-$ are edges of $G$.
\end{lem}

\begin{proof}
We start by fixing $\log n$ disjoint intervals $J_1,\dots,J_{\log n}$ in $[k]$
with each $|J_i| = \Theta(n/\log^2 n)$. We will construct $P$ in steps,
where in step $i$ we expose vertices $v_j$ with $j \in J_i$.
In step 1 we expose all vertices $v_j$ where $j \in J_1$ is even.
Let $S_1$ be the set of exposed neighbours of $x$.
Then $|S_1|$ is hypergeometric with mean
at least $d|J_1|/2n = \omega(\log n)$, 
so by Chernoff bounds $|S_1|=\omega(\log n)$
with probability $1-o(n^{-3})$.

After step $i$ for $i \ge 1$, we have exposed all vertices $v_j$
where $j \in J_{i'}$ with $i'<i$ or $j \in J_i$ is even.
Let $G_i$ be the random induced subgraph of $G$
obtained by deleting all exposed vertices.
Then $\delta(G_i) = d - O(d/\log n)$ 
with probability $1-o(n^{-4})$ 
by \Cref{lem:mindegree}; 
in particular, we can apply \Cref{lem:expansion} to $G_i$.

In step $i+1$ we expose all vertices $v_j$ 
where $j \in J_i$ is odd or $j \in J_{i+1}$ is even,
thus revealing $S_i^- := \{ v_j: v_{j+1} \in S_i\}$
and $T_i := \{ v_j: j \in J_{i+1} \text{ is even}\}$,
which are uniformly random disjoint subsets of $V(G_i)$
of sizes $|S_i|$ and $|J_{i+1}|/2$.
We let $S_{i+1} = T_i \cap \Gamma(S_i^-)$ 
be the set of vertices in $T_i$ with a neighbour in $S_i^-$.
By \Cref{lem:expansion}, with probability $1-o(n^{-4})$ 
we have $|S_{i+1}| \ge \tfrac{d_i|S_i||T_i|}{16n_i}$
where $n_i:=|V(G_i)| \sim n$ and $d_i := \min\{d,n_i/|S_i|\}$.
Note that if $|S_i| \le n/d$ then $|S_{i+1}| \ge |S_i|\log n$, otherwise $|S_{i+1}| \ge |T_i|/16$. 
Thus we reach some $i^*<\log n$ with $|S_{i^*}|=\Omega(|T_{i^*}|)=\Omega(n/\log^2 n)$.

Finally, when exposing $S_{i^*}^-$, the number of exposed neighbours of $y$
is hypergeometric with mean at least $d|S_{i^*}^-|/n = \omega(\log n)$,
so with probability  $1-o(n^{-4})$ we expose some neighbour of $y$,
from which we obtain the required special vertex sequence.
\end{proof}

\begin{figure}[ht]
    \centering
    \begin{tikzpicture}[scale=0.8,main node/.style={circle,draw,color=black,fill=black,inner sep=0pt,minimum width=3pt}]
        \tikzset{cross/.style={cross out, draw=black, fill=none, minimum size=2*(#1-\pgflinewidth), inner sep=0pt, outer sep=0pt}, cross/.default={2pt}}
	\tikzset{rectangle/.append style={draw=brown, ultra thick, fill=red!30}}

	    \foreach \i in {3,...,100}
	    {
	        \node[main node, scale=0.2] (aux) at (\i*0.2-0.6,0){};
	    }
	    \node[rectangle, color=red, scale=0.2] (a) at (0,0) [label=below:$v_1$]{};
	    \node[rectangle, color=red, scale=0.2] (a1) at (1,0){};
	    \node[rectangle, color=red, scale=0.2] (a2) at (3,0){};
	    \node[rectangle, color=red, scale=0.2] (a3) at (5,0){};
	     \node[rectangle, color=red, scale=0.2] (a4) at (7,0){};
	    \node[main node] (b4) at (7.2,0){};
	    
	    \node[main node] (b1) at (1.2,0){};
	    \node[main node] (b2) at (3.2,0){};
	    \node[main node] (b3) at (5.2,0){};
	    
	    \draw[line width= 1 pt, blue] (a) to [bend left=80](b1);
	    \draw[line width= 1 pt, green] (a1) to [bend left=80](b2);
	    \draw[line width= 1 pt, blue] (a2) to [bend left=80](b3);
	    \draw[line width= 1 pt, green] (a3) to [bend left=80](b4);

\draw (9,-2) circle (1);
\node[main node] (y) at ($(9,-2)+(135:1)$)
[label=right:$v_2$]{};
\node[main node] (x) at ($(9,-2)+(160:1)$)
[label=left:$v_2^-$]{};
\draw[thick, blue] (9,-2) ++(160:1) arc (160:495:1);
\draw[line width= 1 pt, blue] (y) to [bend left=0](a4);

\draw[line width= 1 pt,blue] (b3) to (a4);
\draw[line width= 1 pt,blue] (b1) to (a2);
 \draw[line width= 1 pt,blue] (b3) to (a4);

 \draw[line width= 1 pt, green] (a) to (a1);
 \draw[line width= 1 pt, green] (b2) to (a3);

 \draw [decorate,decoration={brace,amplitude=15pt,mirror}] (8,1) -- (0,1) 
    node [midway, above=0.5] {Interval $I_1$};

\draw [decorate,decoration={brace,amplitude=15pt,mirror}] (16,1) -- (8,1) 
    node [midway, above=0.5] {Interval $I_2$};

      \node[scale=2] (dots) at (18,1){$\ldots$};
      \node[scale=1] (dots) at (19,-0.6){$C_1$};
      \node[scale=1] (dots) at (10,-3){$C_2$};
   
    \end{tikzpicture}
    \caption{Incorporating the cycle $C_2$ into the long path on $V(C_1)$. The new path starts at $v_2^-$, following the blue and green paths and then 
     the intervals $I_2,I_3,\ldots, I_k$. In the next step, $v_2^-$ plays the role of $v_1$ and we find the special vertices in $I_2$. }
    \label{fig:rotation}
\end{figure}
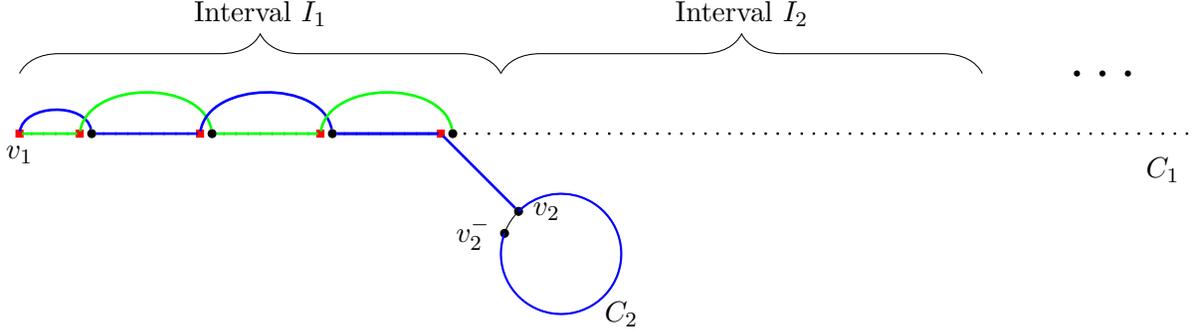

\begin{proof}[Proof of \Cref{thm:regular}]
Let $G$ be an $n$-vertex graph with  $\delta(G)=\omega(\log^3n)$ 
and $\Delta(G)=O(\delta(G))$.
Let $F \sim G_{n,2}$ be a random $2$-regular graph on $V(G)$.
We condition on the lengths of the cycles of $F$, 
fix a $2$-factor $F^*$ on $V(G)$ with precisely these cycle lengths,
and write $F = \pi(F^*)$ where $\pi \in S_n$ is a uniformly random permutation.
By \Cref{lem:permutation properties} we can assume that $F^*$ has $k=O(\log n)$ cycles
$C_1,\dots,C_k$ and at least one cycle $C_1$ of length $\Omega(n)$.
We fix an adjacent pair $v_i^- v_i$ on each $C_i$
and delete $v_1 v_1^-$ from $C_1$ to form a path $P_1$ from $v_1$ to $v_1^-$.
We fix disjoint intervals $I_1,\dots,I_k$ in $P_1$ with each $|I_i|=\Omega(n/\log n)$.
Taking a union bound over all $x,y \in V(G)$ and $i \in [k]$, whp we obtain $F=\pi(F^*)$
satisfying the conclusion of  \Cref{lem:special} for every $x,y$ 
and $G'=G[\pi(I_i)]$ induced by $\pi(I_i)$.

Now we incorporate the cycles $C_2,\dots,C_k$ in turn.
In step $1$, to incorporate $C_2$ we apply \Cref{lem:special}
with $x=v_1$ and $y=v_2$ to obtain a special vertex sequence in $I_1$,
which we use as in \Cref{fig:rotation} to obtain a new path starting at $v_2^-$
that includes $I_2,\dots,I_k$.
In step $i \in [2,k-1]$, to incorporate $C_{i+1}$ we apply \Cref{lem:special}
with $x=v_i^-$ and $y=v_{i+1}$ to obtain a special vertex sequence in $I_i$,
which we use to obtain a new path starting at $v_{i+1}^-$ that includes $I_{i+1},\dots,I_k$.
In step $k$, having incorporated all cycles to form a Hamiltonian path from $v_k^-$ to $v_1^-$,
we apply \Cref{lem:special} with $x=v_k^-$ and $y=v_1^-$ 
to obtain a special vertex sequence in $I_k$, which provides a Hamiltonian cycle.
\end{proof}

\section{Concluding remarks}

There are several possible variations on the hamiltonicity problem
for $G \cup R$ that can be obtained by varying the properties of
the deterministic graph $G$ and the random perturbation $R$.
One such question was recently posed by Espuny D{\'\i}az~\cite{espuny},
who considered the case that $R$ is a random $C_\ell$-factor. Noting that the problem should be easier for larger $\ell$,
he asked whether for any fixed $\ell$ there is some $\alpha=\alpha(\ell)<1/2$ such that if $\delta(G)\geq \alpha n$
and if $F$ is a random $C_\ell$-factor then whp $G\cup F$ is Hamiltonian. Our methods confirm this for large $\ell$
and show moreover that $\alpha$ can tend to zero.

\begin{prop}\label{thm:cycle factor}
For any $\alpha>0$ there exists $k=k(\alpha)$ such that the following holds for all $\ell$ with $k\leq \ell\le n$ and $\ell\mid n$. 
If $G$ is an $n$-vertex graph with $\delta(G)\geq \alpha n$ 
and $G$ is a random $C_\ell$-factor on the same vertex set
then whp $G\cup F$ is Hamiltonian.  
\end{prop}

The main idea of the proof is the same as in our main results, so we will just give a brief sketch here.
Let $k=1000/\alpha^3$, and fix $\ell \in [k,n]$ such that $\ell$ divides n.
The first step is to find a long path $P$ where most vertices are not yet exposed. 
If $\ell> n/\log ^2 n$ we will simply choose $P$ within one of the cycles in $F$,
so let us assume $\ell\leq n/\log^2 n$.
We expose one pair of adjacent vertices in each cycle of $F$, which will call red and blue,
then consider  the auxiliary directed graph whose vertex set is the set of obtained pairs, 
with a directed edge between two pairs if the red vertex of the first pair is adjacent to the blue vertex of the second pair.
By Chernoff bounds this digraph whp has minimum outdegree at least $\alpha n/3\ell$,
so we can greedily find a directed path of this length,
which we use to concatenate the corresponding cycles,
obtaining a path $P$ of length $\alpha n/3$ with $2n/\ell$ exposed vertices.
GIven $P$, the second step of the proof follows the same strategy as before:
expose consecutive pairs $(a,b)$ of unexposed vertices of $P$
and use these to absorb each remaining cycle.
When we consider some such pair $(a,b)$,
assuming that the current graph induced by the unexposed vertices has minimal degree at least $\alpha n/4$,
with probability at least $\alpha^2/32$ the first vertex of $P$ is adjacent to $b$
and $a$ is adjacent to some fixed vertex $v$ in a cycle we wish to absorb.
We have at least $\alpha n/8$ pairs $(a,b)$,
so by Chernoff bounds and our choice of $k(\alpha)$
we whp obtain $n/\ell$ successes before running out of pairs.
Similarly, if $\ell> n/\log ^2 n$ we have $\ell/2$ pairs
and whp obtain $n/\ell$ successes before running out of pairs.

There are many classical results for which one might seek randomly perturbed versions
(e.g. \cite{chvatal1972hamilton,chvatal1972note,jackson1980hamilton,nash1971edge})
but these results typically concern graphs that are quite dense,
whereas it seems more interesting to obtain results on hamiltonicity for sparse graphs.
For example, a recent breakthrough by 
Dragani\'c, Montgomery, Munh\'a Correia, Pokrovskiy and Sudakov~\cite{draganic2024hamiltonicity} 
shows hamiltonicity for constant degree combinatorial expanders, which also implies a conjecture
of Krivelevich and Sudakov~\cite{krivelevich2003sparse} on hamiltonicity of spectral expanders.
Considering hamiltonicity of sparse graphs,
our final open problem is to improve Theorem \ref{thm:regular}
under the additional assumption that $G$ is regular (not just approximately regular).
It may be that no further assumption is required.

\begin{conj}
Let $G$ be an $n$-vertex regular graph and $F \sim G_{n,2}$ 
be a random $2$-regular graph on the same vertex set as $G$.
Then whp $G\cup F$ is Hamiltonian. 
\end{conj}

As evidence towards this conjecture,
note that if $G$ contains a perfect matching $M$
(e.g.~if $G$ is one-regular take $M=G$)
then $M \cup F$ is contiguous to a random cubic graph 
(see \cite[Theorem 9.34]{janson2011random})
so is whp Hamiltonian
(see \cite[Theorem 9.20]{janson2011random}).
If $G$ is regular of even degree
then $G$ contains a $2$-factor (Petersen's Theorem),
so this case would follow from the $2$-regular case;
a first step here is a result of Frieze~\cite{frieze2001hamilton}
that the union of two random $2$-factors is whp Hamiltonian.
For the general conjecture, by Hall's theorem any regular graph $G$
contains a spanning subgraph $H$ where each component is a cycle or an edge,
so it would suffice to show for any such $H$ that whp $H \cup F$ is Hamiltonian.

\end{document}